\documentclass[11pt]{article}

\usepackage[top=2cm, bottom=3cm, left=2cm, right=2cm]{geometry}

\usepackage[abbrev]{amsrefs}

\usepackage{graphicx}
\usepackage{amsmath,
amssymb,
fancybox,enumerate,color,amsthm}

\begin{document}

\renewcommand{\theequation}{\arabic{section}.\arabic{equation}}

\newtheorem{theorem}{\bf Theorem}[section]
\newtheorem{lemma}{\bf Lemma}[section]
\newtheorem{proposition}{\bf Proposition}[section]
\newtheorem{corollary}{\bf Corollary}[section]
\newtheorem{remark}{\bf Remark}[section]
\newtheorem{example}{\bf Example}[section]
\newtheorem{definition}{\bf Definition}[section]

\newcommand{\supp}{\mathop{\mathrm{supp}}}
\newcommand{\PA}{\partial}
\newcommand{\ve}{\varepsilon}
\newcommand{\TR}{\textcolor{red}}
\newcommand{\TB}{\textcolor{blue}}
\newcommand{\ds}\displaystyle

\newcommand{\re}{{\Bbb R}}
\newcommand{\nid}{\noindent}

\newcommand{\nn}{\nonumber}
\newcommand{\hu}{\hat{u}}
\newcommand{\tw}{\widetilde{w}}
\newcommand{\D}{D_0^{1,2}}
\newcommand{\DK}{D_{0,K}^{1,2}}
\newcommand{\rhp}{\rightharpoonup}
\def\R{{\mathbb R}}
\newcommand{\la}{\langle}
\newcommand{\ra}{\rangle}
\newcommand{\ls}{\{}
\newcommand{\rs}{\}_{n=1,2\cdots}}
\newcommand{\hv}{\hat{v}}
\newcommand{\hvp}{\widehat{\vp}}
\newcommand{\hw}{\widehat{w}}
\newcommand{\cv}{\bar{v}}
\newenvironment{prf}[1]
   {{\noindent \bf Proof of {#1}.}}{\hfill \qed}
\newcommand{\eq}[1]{{\begin{equation}#1\end{equation}}}

%%%%%%%%%%%%%%%%
\makeatletter
 \@addtoreset{equation}{section}
\makeatother
\def\theequation{\arabic{section}.\arabic{equation}}
%%%%%%%%%%%%%%%%

\title{Non-uniqueness of mild solutions for 2d\,-\,heat equations \\ with singular initial data}

\author{\small
	Yohei Fujishima\\
	\small
	Department of Mathematical and Systems Engineering,
	Faculty of Engineering,
	\\
	\small
	Shizuoka University, \\
	\small
	3-5-1 Johoku, Hamamatsu 432-8561, Japan, 
	%\\
	%\small
	%{fujishima@shizuoka.ac.jp}
	\vspace{5pt}\\
	\quad\\
	\small
	Norisuke Ioku
	\\
	\small
	Mathematical Institute, 
	\\
	\small
	Tohoku University, \\
	\small
	Aramaki 6-3, Sendai 980-8578, Japan,
	%\\
	%\small
	%{ioku@tohoku.ac.jp}
	\vspace{5pt}\\
	\quad\\
	\small
	Bernhard Ruf
	\\
	\small
	Accademia di Scienze e Lettere - Istituto Lombardo,\\
	\small
	via Borgonuovo, 25, Milano 20121, Italy,
	%\\
	%\small
	%{bruf001@gmail.com}
	\\
	\quad\\
	{\small and}\\
	\quad\\
	\small
	Elide Terraneo\footnote{Corresponding author: elide.terraneo@unimi.it (E. Terraneo)}\\
	\small
	Dipartimento di Matematica ``F. Enriques'',\\
	\small
	Universit\`a degli Studi di Milano,\\
	\small
	via C. Saldini 50, Milano 20133, Italy.
	%\\
	%\small
	%{elide.terraneo@unimi.it}
}

%\author{Yohei Fujishima, Norisuke Ioku, Bernhard Ruf and Elide Terraneo}
\date{
}
\maketitle

\abstract{In a recent article by the authors \cite{FIRT1} it was shown that wide classes of semilinear elliptic equations with exponential type nonlinearities admit singular radial solutions $U$ on the punctured disc in $\mathbb R^2$ which are also distributional solutions on the whole disc. We show here that these solutions, taken as initial data of the associated heat equation, give rise to non-uniqueness of {\it mild solutions}: ${u_s}(t,x) \equiv U(x)$ is a stationary solution, and there exists also a solution ${u_r}(t,x)$ departing from $U$ which is bounded for $t > 0$. While such non-uniqueness results have been known in higher dimensions 
by
Ni--Sacks~\cite{NS}, Terraneo~\cite{Te} and Galaktionov--Vazquez~\cite{GV}, only two very specific results have recently been obtained in two dimensions by Ioku--Ruf--Terraneo \cite{IRT1} and Ibrahim--Kikuchi--Nakanishi--Wei~\cite{IKNW}.
}

\medskip
\noindent
{\small
	{\bf Keywords}: 
	Semilinear heat  equations, non-uniqueness , 2-dimensions, exponential nonlinearities
	\vspace{5pt}
	\newline
	{\bf 2020 MSC}: Primary: 35K58; Secondly: 35A01, 35A02
	\vspace{5pt}
}

\section{Introduction}

In this paper we deal with  the Cauchy problem for a semilinear heat equation
\begin{equation}\label{eq:1.1}
\left\{
\begin{aligned}
&\partial_t u-\Delta u
=f(u)
&&  \text{in}\ \ (0,T)\times B_R,\\ \ \ 
&u(t,\cdot) = 0 \ &&  \hbox{on } (0,T)\times \partial B_R, \\
&u(0,x)
=u_0(x)\geq 0 && \text{in}\ \ B_R,
\end{aligned}
\right.
\end{equation}
where
 $B_R\subset \R^2$ is a  ball of radius $R>0$ centered at the origin,  $f\in C^1([0,\infty))$, $u(t,x):[0,T)\times B_R\to [0,\infty)$ is the unknown	function
and $u_0$ is a  given initial data.
 We prove, for a large class of nonlinearities and for some particular unbounded data, that, at least, two {\it mild} solutions exist. The unbounded data  we consider are some {\it singular  positive  solutions} that solve the
 equation
  \begin{equation}\label{eq:poli}
 -\Delta u=f(u).
 \end{equation}
 By {\it singular positive solution} we mean  a positive  function $U\in C^2(B_R\setminus \{0\})$, radially symmetric, with  $\lim_{x\to 0}U(x)=\infty$, that satisfies \eqref{eq:poli} in $B_R\setminus \{0\}$ in the classical sense.

  We  first recall some known results  for specific nonlinearities in $\R^N$, with $N\geq 3$. 
In $\R^N$, with $N\geq 3$, 
  for $f(u)=u^p$, $p>1$,  the existence of  singular  solutions to \eqref{eq:poli} is well-established. The first result is due to Lions     in \cite{L}, who obtained the existence of a singular solution on $B_R\setminus \{0\}$, for $1<p<\frac{N}{N-2}$. This solution behaves like the fundamental solution of the Laplace equation, i.e.  $U(x)\sim|x|^{2-N}$ as $x\to 0$.
 For  $\frac{N}{N-2}\leq p<\frac{N+2}{N-2}$, Ni--Sacks~\cite{NS} (see also Aviles~\cite{A} and Chen--Lin~\cite{CL}) proved the existence of infinitely many singular solutions to \eqref{eq:poli}. Moreover,   $U(x) \sim |x|^{2-N}(-\log|x|)^{\frac{2-N}2}$ near $0$, if $p=\frac{N}{N-2}$, and  $U(x) \sim |x|^{-\frac{2}{p-1}}$ near $0$, if $\frac{N}{N-2}<p<\frac{N+2}{N-2}$. Finally, the case of $p\geq\frac{N+2}{N-2}$ has been treated in  \cite{CGS}.
Moreover, these singular solutions  satisfy the equation~\eqref{eq:poli} in the non-punctured ball in the sense of distributions only  for  $p \geq \frac{N}{N-2}$.
 
In  the particular  case $p=\frac{N}{N-2}$,  Ni--Sacks \cite{NS}  proved that the Cauchy problem \eqref{eq:1.1}  with initial data $u_0(x)=U(x)$ admits at least two solutions:
 the  stationary singular solution  $u_s(t,x)=U(x)$, $t>0$  and a regular  solution which is bounded for positive time. Thus non-uniqueness occurs. Terraneo \cite{Te}  extended this result to $\R^N$
 (see also 
\cite{MT,Ta}).
 
We mention another non-uniqueness result in $\R^N$, with $N\geq 3$, in the case $f(u)=u^p$, with $p >\frac{N}{N-2}$. In $\R^N$  the function  
$V_p(x)=L_p\, |x|^{-\frac{2}{p-1}}$ with $
L_p
:=
\big\{
\frac{2}{p-1}\big(N-2-\frac{2}{p-1}\big)
\big\}^{\frac{1}{p-1}}, $
is an exact singular solution to the  elliptic equation \eqref{eq:poli}  and it is 
  also a  distributional solution in $\R^N$.
   Using this explicit singular solution,  Galaktionov--Vazquez \cite{GV}
   proved that with $u_0(x)=V_p(x)$ the Cauchy problem \eqref{eq:1.1} in $\mathbb{R}^N$ admits a 
	classical solution  $u_r(t,x)$
	that converges to $V_p(x)$ in the sense of distributions as $t\to 0$,  for any $p>\frac{N}{N-2}$, if $3\leq N\leq 10 $, and for $\frac{N}{N-2}<p<1+\frac{4}{N-4-2\sqrt{N-1}}$, if $N\geq 11$. Thus, the Cauchy problem \eqref{eq:1.1} in $\mathbb{R}^N$ with $f(u) = u^p$ and $u_0(x)=V_p(x)$ has (for these $p$) at least two different solutions:
the stationary solution $u_s(t,x)=V_p(x)$ and the regular solution $u_r(t,x)$.
Other type of non-uniqueness results regarding to self-similar solutions can be found in \cite{HW}.
Moreover, we remark that the singularity of $V_p$ is in some sense a threshold for existence and non-existence of a solution to \eqref{eq:1.1}. 
Indeed, 
 $V_p$ belongs to $L^{\frac{N(p-1)}2,\infty}$ and it is known that for any $u_0\in L^q$ for either $q=\frac{N(p-1)}2>1$ or $q>\frac{N(p-1)}2$ and $q\geq 1$, then problem \eqref{eq:1.1} has a local-in-time solution, while  for $1\leq q<\frac{N(p-1)}2$ there exists a nonnegative initial function  $u_0\in L^q$ 
 such that problem \eqref{eq:1.1} has no nonnegative
local-in-time solutions  (see \cite{W1}, \cite{W3}). 
Furthermore, 
it was proved by Baras--Pierre~\cite{BP} and Robinson--Sier\.{z}\polhk{e}ga~\cite{RS}
that 
for $p>1+\frac{2}{N}$ and the initial data $u_0=\varepsilon V_p$, 
problem 
\eqref{eq:1.1} has a local-in-time solution if $\varepsilon $ is sufficiently small, 
while
problem 
\eqref{eq:1.1} has no nonnegative local-in-time solutions if $\varepsilon $ is sufficiently large.
Some developments in these directions can be found in 
\cite{FI2,FI3,FI4,FHIL,HI01,IKO01,IKO02,Na,Na2,S,SW}.

  Galaktionov--Vazquez \cite{GV} consider also  exponential nonlinearities. They remark that
$U(x)=-2\log|x|$ solves  the elliptic equation \eqref{eq:poli} with
$f(u)=2(N-2)\,e^u$ on $B_1\setminus\{0\}$   and they prove a similar  non-uniqueness result for the corresponding evolution equation with initial data $U(x)$ and Dirichlet boundary conditions. 
 
\par 
Less is known for the  two dimensional case $N=2$. For $f(u)=u^p$ with $p>1$ or $f(u)=e^u$ 
there exist singular solutions to the elliptic equation
\eqref{eq:poli} (see Lions \cite{L} and Tello \cite{T}). Such solutions behave like the fundamental solution $U(x)\sim-\log|x|$ and they are not distributional solutions on the non-punctured ball. Indeed,  it is proved in \cite{DGP} that for  nonlinearities with growth lower or equal to $e^u$, any distributional solution is regular.
On the other hand, for $f(u)\sim e^{u^q}$, with $q>1$,
any singular solution has a singularity weaker than the fundamental solution $-\log|x|$ and it is distributional.
We recall two results for specific nonlinearities $f(u)\sim e^{u^2}$.
In \cite{FR} it was noted that 
\begin{equation}\label{exp} 
u(r) = (-2\log r)^{1/2} \ \hbox{ solves the equation  \eqref{eq:poli} for }\ f(s) = \frac 1{s^3}\ e^{s^2}.
\end{equation}
Using  this solution, Ioku--Ruf--Terraneo in \cite{IRT1} show that there exists $R > 0$ such that the boundary value problem
\begin{equation*}\left\{
\begin{aligned}{}  -\Delta u &= g(u)  &&\hbox{in} \ B_R\setminus\{0\},
\\
 u &= 0   &&\hbox{on} \ \partial B_R,
\end{aligned} \right.
\end{equation*}
has a singular radial and positive solution $U(x)$ with growth $(-2\log|x|)^{1/2}$ near $0$, for the nonlinearity
$$
g(s) = \left\{ \begin{aligned}{} &e^{s^2}{/s^3},\  \ &&s\ge b, \vspace{0.2cm}\\
&as^2,	\  \ &&0 \le s \le b,
\end{aligned}\right. \quad  \text{where}\ a, b\ \hbox{are constants such that }\ g \in C^1\bigl(
[0,\infty) \bigr).
$$
This solution is also a distributional solution 
on the non-punctured ball $B_R$.
Then, they consider the Cauchy problem \eqref{eq:1.1} with nonlinearity $g(u)$ and initial data $u_0(x)=U(x)$  and they prove the existence of two mild solutions.
They also prove that, 
for the initial data 
$u_0=\lambda U$ with $0<\lambda<1$,
%	strictly below the threshold  $U$,  
problem  \eqref{eq:1.1} is well-posed, while  for the initial data 
$u_0=\lambda U$ with $\lambda>1$,
% above this function $U$, 
 there exists no non-negative solution.

Independently,  Ibrahim--Kikuchi--Nakanishi--Wei~\cite{IKNW} consider the  nonlinearity $f_m(u)$, such that  $\lim_{u\to \infty}\frac{f_m(u)}{u(e^{u^2}-1)}=1$ and 
$\lim_{u\to 0}\frac{f_m(u)}{m u}=-1$ with $m\in \R$, $m>0$, and they prove the existence of a singular positive radial solution $U$ to
$$-\Delta u=f_m(u),$$ on $\R^2$, for some particular $m>0$.
Moreover, 
 $U(x)
=
 \big(-2\log|x| - 2\log(-\log|x|)-2\log2\big)^{1/2}
 +o(1)
 $ as $x\to 0$, and $\lim_{|x|\to\infty}U(x)=0$. They also prove  a result of non-uniqueness to problem  \eqref{eq:1.1} with nonlinearity $f_m$ and initial data $u_0(x)=U(x)$ in $\R^2$.
 
 Recently, in \cite{FIRT1} we observed (see also \cite%[Lemma~5.2]
 {GG}) that there exist explicit 
 singular solutions to
 \begin{equation}\label{G} -\Delta u = g(u)
 \end{equation} 
 for certain specific nonlinearities
 with growth $g(s) \sim e^{s^q}, q > 1$
or 
$g(s)\sim e^{e^s}$.
 Indeed, for 
 $$ g(s) = \frac 4{q\,q'}\frac{e^{s^q}}{s^{2q-1}}\ , \ q > 1 \ , \ 
 \hbox{ 
the functions $v(x) = (-2\log|x|)^{1/q}$ solve
 \eqref{G}
 near the origin,
 }
 $$
 while for 
 $$ g(s) = 4\, \frac {e^{e^s}}{e^{2s}} \ , \ 
 \hbox{
  the function $v(x) = \log(-2\log|x|)$ solves \eqref{G} near the origin.
% the solution to \eqref{G} is} \quad .
}
 $$
 Starting from these explicit singular solutions, in the  recent paper \cite{FIRT1} we were also able to construct a singular solution to \eqref{eq:poli}
  for  a wide class of nonlinearities  of exponential type. In order to classify the nonlinearities we deal with, let us consider functions $f \in C^1([0,\infty))$ such that $f(s)>0$
%   ($\forall s>0$) 
   and $f'(s)>0$ 
%   ($\forall s>0$)
for any $s>0$, and assume that $f$ satisfies the following assumption:
\begin{itemize}
	\item[($f1$)]
	\label{f1}	
\qquad 	
%	\[
$\displaystyle	
F(s):=\int_s^{\infty}\frac{1}{f(\tau)}d\tau<\infty, \ \ \ \ \ s>0.
$
%	\]
\end{itemize}
Furthermore, we introduce
%\qquad 
\[
\displaystyle \frac 1{B_1[f](s)} := (-\log F(s))[1-f'(s)F(s)]
\] 
%\par \smallskip\noindent
and its de l'Hospital form 
%\qquad 
\[
 \displaystyle \frac 1{B_2[F](s)} := \frac{\big(1-f'(s)F(s)\big)'}{\big(\frac 1{-\log F(s)}\big)'}
 \]
and assume that
\par \medskip
\begin{itemize}
	\item[($f2$)]
\label{f2}
\qquad 	
the limits \
	$\displaystyle \lim_{s\to \infty}f'(s)F(s)$\ \
	and \
	$\displaystyle \lim_{s\to \infty}B_2[f](s)$\ \
	exist, and satisfy
	\begin{equation}
	\label{eq:1.2cc}
	\displaystyle \lim_{s\to \infty}f'(s)F(s)=1
	\end{equation}
	and
	\begin{equation}
	\label{eq:1.2c}
	\displaystyle B:=\lim_{s\to \infty}B_2[f](s)\in [1,\infty).
	\end{equation}
\end{itemize}
We also  need  a further hypothesis which 
enables us to construct a singular solution.
For $f$ satisfying $(f1)$ and $(f2)$, and with $B$ as in \eqref{eq:1.2c}, let $B'$ be its dual exponent, and choose the associated ``model nonlinearity'' $g$ as follows:  
\begin{equation}\label{g}
g(s):= \hbox{$\frac{4}{BB'}s^{1-2B'}$}\, e^{s^{B'}} \  \hbox{if } B > 1, \quad g(s):=4\,\frac{e^{e^s}}{e^{2s}} 
\ \ \hbox{if } B = 1,
\end{equation}
and, respectively, 
\begin{equation}\label{V}
v(x)=(-2\log|x|)^\frac1{B'} \ \ \hbox{if } B > 1,\ \ v(x)=\log(-2\log|x|)\ \ \hbox{if } B = 1.
\end{equation}
Define 
\[
G(s):=\int_s^{\infty}\frac{1}{g(\tau)}d\tau
\]
for $s>0$. It is remarkable that 
$G(v)$ can be expressed in one form for each cases $B> 1$ and $B=1$
by
\[
G\bigl(v\bigr)=\frac{B}{4}|x|^2\left(1+\log\frac{1}{|x|^2}\right).
\]
Setting
\begin{equation}\label{eq:1.5}
\begin{aligned}
R_1(x)
:=
\Big|\frac{1}{B_1[f]\big(\tilde u(x)\big)}-\frac{1}{B_1[g]\bigl(v(x)\bigr)}\Big|\ ,\quad 
R_2(x)
:=
\Big|\frac{1}{B_2[f]\bigl(\tilde u(x)\bigr)}-\frac{1}{B_2[g]\bigl(v(x)\bigr)}\Big|\ ,
\end{aligned}
\end{equation}
where \begin{equation}\label{eq:u}
\tilde u(x):=F^{-1}\Bigl[G[v(x)]\Bigr], \ 
\end{equation}
we make the assumption
\par \smallskip \noindent
\begin{itemize}
	\item[($f3$)]
\label{f3}\qquad 
$\displaystyle \lim_{x \to 0}(-\log |x|)^{\frac{1}{2}}\left[R_1(x)+R_2(x)\right]=0$.
\end{itemize}
\par \bigskip \noindent
Then, one has the following result.

\begin{theorem}[{\cite[Theorem~2.2]{FIRT1}}]\label{theorem:1.1}
	Let	$f\in 
	C^1\bigl([0,\infty)\bigr)$,
	$f(s)>0$, $f'(s)>0$ for all $s>0$. Assume that  $f$ satisfies
	(f1), (f2) and (f3) with respect to 
	the model nonlinearity $g$.
	Then, there exist $R>0$ and 
	$U\in 
	C^{2}(B_R\setminus\{0\})
	$ 
	satisfying

\begin{equation}\label{eq:1.1aaa}
	\left\{
	\begin{aligned}
	-\Delta u
	&
	=f(u)
	&& \text{in}\ B_R\setminus\{0\},
	\\
	u
	&
	=0
	&& \text{on}\ \partial B_R,
	\end{aligned}
	\right.
	\end{equation}
	in the classical sense in $B_R\setminus\{0\}$.
	Furthermore,
	\begin{equation*}
\lim_{x\to 0}	U(x)
	=\infty
	\end{equation*}
	and $U$
	 satisfies
	$-\Delta U=f(U)$ in $B_R$
	in the sense of distributions.
\end{theorem}
The key tool of the proof of Theorem \ref{theorem:1.1} is the following. Let $f$ be a function satisfying the hypothesis of the Theorem \ref{theorem:1.1}. The idea is to connect $f$ to the explicit nonlinearity $g$ as \eqref{g},  which is such that
$-\Delta v=g(v)$ admits the explicit solution $v(x)$ given by \eqref{V}.
Then, thanks to the result  by Fujishima--Ioku \cite{FI}, one has that for general $f$ and $g$
 \[
v(x) \ \hbox{ is a solution to }\ -\Delta v=g(v)
\] 
if and only if
\begin{equation*}
\tilde u(x):=F^{-1}\Bigl[G[v(x)]\Bigr] \ 
\end{equation*}
satisfies
\begin{equation}\label{eq:transformation}
-\Delta \tilde u=f(\tilde u)
+\frac{|\nabla{\tilde u}|^2}{f(\tilde u)F(\tilde u)}\Big[ g'(v)G(v)-f'(\tilde u)F(\tilde u)\Big].
\end{equation}
The assumptions 
$(f1)$--$(f3)$
serve to show that the second term on the right
hand side
in \eqref{eq:transformation} is suitably small. Then, by a fixed point argument, one obtains a solution $U$ to $-\Delta u = f(u)$ which is close to $\tilde u$.
Therefore, thanks to 
the transformation~\eqref{eq:u} and
the transformed equation~\eqref{eq:transformation},
one can connect a general nonlinearity $f$
to a {\it model nonlinearity} $g$ and obtain an approximate solution $\tilde u$ which then yields a singular solution $U$ of $-\Delta u=f(u)$ close to  $\tilde u$.
\par \medskip
It was shown in \cite{FIRT1} that, despite the somewhat involved hypothesis $(f3)$, wide classes of nonlinearities are covered by this theorem, e.g.
$$ f(s) = s^r\, e^{s^q} \ (q > 1\ , r \in \mathbb R)\ ;\quad f(s) = e^{s^q + s^r}\ \ (q > 1,\ 0<r<\frac q2\ \ {\rm or}\ \  1<q<4,\ r=q-1)  \ ; \quad f(s) = e^{e^{s}}.  \  
$$
\par \medskip
\begin{remark}
Since our argument in \cite{FIRT1} is a direct construction by contraction map argument, 
the explicit asymptotic behavior of the singular solution in Theorem~\ref{theorem:1.1} 
was also obtained as follows:
\begin{equation}\label{eq:1.2e}
U(x)
=
\tilde u(x)
+
O\Big(f\bigl(\tilde u(x)\bigr)F\bigl(\tilde u(x)\bigr)\sup_{|y|\le |x|}
\bigl(R_1(y)+R_2(y)\bigr)\Big)
\qquad \text{as}\  |x|\to 0,
\end{equation}
where $	\tilde u(x)=F^{-1}\Bigl[G[v(x)]\Bigr]$ and $v(x)$ is as 
\eqref{V}.
We should mention that Kumagai~\cite{K} recently proved existence of a singular solution with very wide class of nonlinearities 
$f(u)$ in $N=2$ by an indirect proof.  However,  his argument 
does not derive the asymptotic behavior of the solution.

\end{remark}
Furthermore, 
the singular solution $U$ satisfies the following:
\begin{proposition}
\label{est}
	For any $\sigma >0$, there exists $r_\sigma$, such that
	\begin{equation}\label{eq:f}
	\frac{C}{|x|^2 \left(1-2\log|x|\right)^{1+\frac{1}{B}+\sigma}}
	\le 
	f\bigl(U(x)\bigr)\le
	\frac{C}{|x|^2 \left(1-2\log|x|\right)^{1+\frac{1}{B}-\sigma}},
	\end{equation}
	\begin{equation}\label{Stima1}
	0\leq  U(x)\leq (-2\log|x|)^{1-\frac 1B+\sigma},
	\end{equation}
	and
	\begin{equation}\label{Stima2}
	|\nabla U(x)|\leq \frac{1}{|x|(-2\log|x|)^{\frac 1B-\sigma}}
	\end{equation}
	for all $x$ with
	$|x|\leq r_\sigma$. 
Furthermore, it holds that
\begin{equation}\label{eq:ioku1.17}
F(U)=F(\tilde u)\bigl(1+o(1)\bigr)=\frac{4}{B|x|^2(1-2\log |x|)}\bigl(1+o(1)\bigr)
\ \text{as}\ |x|\to 0.
\end{equation}
\end{proposition}
\begin{proof}
Since \eqref{eq:f}--\eqref{Stima2} is already proved in 
\cite[Lemmas 3.5 and 3.6]{FIRT1}, we only prove \eqref{eq:ioku1.17}.
Let $\eta$ be the remainder term given by $U=\tilde u +\eta$. 
Then, Taylor's theorem yields that
\[
F(U)
=
F(\tilde u)-\frac{\eta}{f\bigl(\tilde u+c \eta\bigr)}
=
F(\tilde u)\left(1-\frac{\eta}{F(\tilde u)f\bigl(\tilde u+c \eta\bigr)}\right)
\]
for some $c\in (0,1)$. 
Meanwhile, it holds by \cite[(3.19) and (3.20)]{FIRT1} that
\[
\lim_{x\to 0}\frac{f\bigl(\tilde u\bigr)}{f\bigl(\tilde u+c\eta\bigr)}=1.
\]
Therefore, by \eqref{eq:1.2e}, we have
\[
\left|
\frac{\eta}{F(\tilde u)f\bigl(\tilde u+c \eta\bigr)}
\right|
\le 
C
\frac{|\eta|}{F(\tilde u)f\bigl(\tilde u\bigr)}
\le
\tilde C 
\sup_{|y|\le |x|}
\bigl(R_1(y)+R_2(y)\bigr)
\]
for some $C,\tilde C>0$. 
The conclusion holds since
the most right hand side of the above inequality converges to 0 as $x\to 0$
and 
$F(\tilde u)=G(v)=\frac{4}{B|x|^2(1-2\log |x|)}$.
\end{proof}

\begin{remark}\label{RM12}
	As a consequence of Proposition~\ref{est}, we have $U\in L^p$ for all $p\in [1,\infty)$ and
	$f(U)\in L^1(B_R)$ and $f(U)\notin L^p(B_R)$
	for any $p\in (1,\infty]$.
	Furthermore,
	if $1<B<2$ then there exists small $r_0>0$ such that the singular solution 
	$U$ belongs to the energy class $H^1(B_{r_0})$.
\end{remark}

\section{Main results}
In the present paper we will prove that the Cauchy problem \eqref{eq:1.1} with $U$, given by Theorem \ref{theorem:1.1}, as initial data has {\it two mild solutions}. 
\par \smallskip
In order to present the result, let us start by introducing the functional framework and the definition of mild solutions.
Let $Y$ be the Fr\'echet space defined by
\[
\displaystyle Y:=\bigcap_{1\le p <\infty}L^p(B_{R})
\]
endowed with the metric
\[
\text{dist}(u,v) 
=
\sum_{k=1}^{\infty}2^{-k}
\frac{\|u-v\|_{L^k}}{1+\|u-v\|_{L^k}}.
\]
It is known that $u(t)\to u_0$ in $Y$ if and only if $u(t)\to u_0$ in $L^p$ for all $1\le p <\infty$. 
 We introduce the definition of  local in time 
mild solution  to  problem  \eqref{eq:1.1}.
For the sake of simplicity, in the following we will omit the underlying space $B_R$.

\begin{definition}[mild solution]
For $u_0\in Y$, we say $u(t)$ is a mild solution  ($Y$-mild solution) of problem~\eqref{eq:1.1}
if 
$u\in C([0,T);Y)$
and
\begin{equation}\label{eq:6.1}
u(t)={\rm e}^{t\Delta}u_0+\int_0^t {\rm e}^{(t-s)\Delta}f\bigl(u(s)\bigr)ds\quad \text{in} \ L^p(B_R)\ \text{for\ all\ }1\le p <\infty.
\end{equation}
\end{definition}

\par \medskip

First, we   construct 
a mild solution  $u_r(t,x)$, that is also {\it regular}, namely bounded for positive time,
to 
 problem \eqref{eq:1.1}
with the initial data $U$, where $U$ is the singular radial and positive solution  found in Theorem \ref{theorem:1.1}. 
\begin{proposition}\label{extcl} Assume that	$f\in 
	C^1\bigl([0,\infty)\bigr)$,
	$f(s)>0$, $f'(s)>0$ for all $s>0$ and  that  $f$ satisfies
	(f1), (f2) and (f3) with respect to 
	the model nonlinearity $g$. Let $U$ be the singular solution  to equation \eqref{eq:1.1aaa} constructed in Theorem~\ref{theorem:1.1}.
	Then, there exists some positive time $T=T(U)>0$ and a function $u\in{C}([0,T);Y)\cap L^{\infty}_{loc}(0,T;L^\infty)$ which is a mild  solution to problem \eqref{eq:1.1}	with initial data 
	$u_0=U$.
\end{proposition}
 Then, we will prove that also the stationary solution $u_s(t,x)=U(x)$, for $t>0$, is  a mild solution.
As a consequence, the non-uniqueness of mild solutions holds as follows.

\begin{theorem}\label{nonun}
Assume that	$f\in 
	C^1\bigl([0,\infty)\bigr)$,
	$f(s)>0$, $f'(s)>0$ for all $s>0$ and  that  $f$ satisfies
	(f1), (f2) and (f3) with respect to 
	the model nonlinearity $g$. 
Let 
$U$
be the singular solution to equation \eqref{eq:1.1aaa} constructed in Theorem~\ref{theorem:1.1}. Then, for some positive time $T=T(U)$, there exist  at least two  mild solutions of problem \eqref{eq:1.1} with initial data 
	$u_0=U$.
\end{theorem}

\begin{remark}\label{Remark:2.1}
	We expect that as in  \cite{IRT1}  the singular solution $U$ plays the role of threshold between existence and nonexistence. More precisely, for the initial data 
$u_0=\lambda U$ with $0<\lambda<1$,
%	strictly below the threshold  $U$,  
problem  \eqref{eq:1.1} is well-posed, while  for the initial data 
$u_0=\lambda U$ with $\lambda>1$,
% above this function $U$, 
 there exists no non-negative solution. See also \cite{FI},  \cite{HM} for similar results. 
	\end{remark}

\par \medskip

 The question of uniqueness of solution for the Cauchy problem of partial differential equations has been widely studied.  Recently, some non-uniqueness results of  weak non-steady  solutions to the  Euler equation have been
proposed by
 De Lellis--Sz\'ekelyhidi \cite{DS}
 by using convex integration techniques. Then, these 
	 results have been  adapted to the framework of Navier--Stokes equations on the torus in three dimensions by Bruckmaster--Vicol \cite{BV} who proved non-uniqueness of weak solutions with finite kinetic energy. We also mention the results of  Luo \cite{Lu} and Lemari\'e-Rieusset \cite{LR} who proved the existence of steady non-trivial  weak solutions  for the Navier--Stokes equations on the torus, respectively, in dimension greater 
than
or equal to  4 and in dimension 2. As a consequence, they obtained  some  non-uniqueness results in a class of weak solutions for Navier--Stokes equations.
	 
\par \medskip	 
The plan  of the  paper is the following.  In Section 3 we describe the behavior of a concave function along the solution of the heat equation on a ball with Dirichlet boundary condition. We also present Perron's monotone method. Both these results  will be essential to prove the existence of a regular mild solution  to the Cauchy problem \eqref{eq:1.1} with initial data $u_0=U$, where $U$ is the singular solution of the associated elliptic equation.
Section 4 is devoted to the proof of the existence of a regular  mild solution (Proposition \ref{extcl}).
Finally, in Section 5 we will prove the non-uniqueness 
stated in
Theorem \ref{nonun}.

\section{Some useful tools}

We recall some  results concerning the solution of the heat
equation on a ball with Dirichlet boundary conditions (see Appendix B in \cite{QS}). Let us denote
by $\textrm{e}^{t\Delta}$ the Dirichlet heat semigroup in an open ball $B_R$ with radius $R>0$. It is known
that for any $\varphi \in L^p(B_R)$, $1\leq p\leq \infty$, the function
$u=\textrm{e}^{t\Delta}\varphi$ solves the heat equation $u_t-\Delta u=0$ in $
(0,\infty)\times B_R$ and $u\in C((0,\infty)\times \overline{B_R})$,
$u=0$ on $ (0,\infty)\times \partial B_R$. Moreover, there exists a
positive
${C}^\infty$
function $G: (0,\infty)\times B_R\times
B_R \to \R$ (the Dirichlet heat kernel) such that
$$
{\rm e}^{t\Delta }\varphi(x)=\int _{B_R}G(t,x,y) \varphi(y)dy,
$$
for any $\varphi \in L^p(B_R)$, $1\leq p\leq \infty$. We prepare
a basic lemma.

\begin{lemma}\label{jensen}
	Let $\varphi: B_R \to [0,\infty)$ be a measurable function and
	$H:[0,\infty) \to \mathbb{R}$ be a concave function  such that
	$H(0)\geq0$. Then
	$$
	{\rm e}^{t\Delta
	}H\left(\varphi\right)\leq
	H\left({\rm e}^{t\Delta }\varphi\right).
	$$
\end{lemma}

\begin{proof}
 Let $H$ be a concave function and $\varphi\geq
0$ be a measurable function.  By  Jensen's inequality, denoting
$\overline{ G}=\overline{ G}(t,x)=\int _{B_R}G(t,x,y) dy>0$,
we
obtain
\begin{equation*}
H\Big(\frac1 { \overline G(t,x)}\int_{B_R}G(t,x,y) {\varphi(y)}
{dy}\Big)\geq \frac1 { \overline G(t,x)}\int_{B_R}G(t,x,y)H\big(
{\varphi(y)}\big){dy}.
\end{equation*}
Therefore
\begin{equation}\label{I}
H\Big(\frac{{\rm e}^{t\Delta }\varphi}{\overline G
}\Big)\geq\frac1 { \overline G}\ {\rm e}^{t\Delta
}H\left({\varphi}\right).
\end{equation} Moreover,  since  $H$ is concave on $[0,\infty)$,
$H(0)\geq0$,  and $\overline G(t,x)\leq 1$, for any $x\in B_R$ and
$t>0$, we have
$$
H(s)= H\Big(\overline G \frac{s}{\overline G}+(1- \bar G)0 \Big)\geq\overline G\ H\left(\frac{s}{\overline G}\right)+(1-{\overline G}) H(0)
\geq\overline G\ H\left(\frac{s}{\overline G}\right).
$$
So  for $s={\rm e}^{t\Delta }\varphi$ we get
\begin{equation}\label{II}
\frac{H(\textrm{e}^{t\Delta}\varphi)}{\overline G}\geq H\Big(\frac{{\rm
		e}^{t\Delta }\varphi}{\overline G  }\Big).
\end{equation}
Finally, \eqref{I} and \eqref{II} imply  the desired inequality
$\displaystyle
H\left({\rm e}^{t\Delta }\varphi\right)\geq {\rm e}^{t\Delta
}H\left({\varphi}\right).
$
\end{proof}

The next proposition plays an  essential role in our argument.
\begin{proposition}\label{Perron}
	Let $u_0\in L^1(B_R)$, $u_0\geq 0$ and $\bar{u}\in{C}^{1,2}((0,T)\times B_R)\cap{C}((0,T)\times \overline{B_R}) $, $\bar{u}(t,x)\geq 0, $ for any $(t,x)\in (0,T)\times B_R$, satisfying 
	\begin{equation}\label{int0}
	\bar{u}(t,x)\geq {\rm e}^{t\Delta}u_0+\int_{0}^t{\rm e}^{(t-s)\Delta}f(\bar{u}(s))\, ds
	\end{equation}
	 for almost every $(t,x)\in (0,T)\times B_R$. Then, there exists a function $u\in{C}^{1,2}((0,T)\times B_R)\cap{C}((0,T)\times \overline{B_R}) $
	satisfying
	\begin{equation}\label{int1}
	{u}(t,x)={\rm e}^{t\Delta}u_0+\int_{0}^t{\rm e}^{(t-s)\Delta}f({u}(s))\, ds
	\end{equation}
	and 
	\begin{equation}\label{diff1}
	\partial_tu=\Delta u+f(u)
\end{equation}
	for any $(t,x)\in (0,T)\times B_R$. Moreover,  $0\leq u(t,x)\leq\bar{u}(t,x)$ for any $(t,x)\in (0,T)\times B_R$, $u(t,x)=0$ for any $(t,x)\in (0,T)\times\partial B_R$.
\end{proposition}
The previous result
% in $\mathbb{R}^N$
is 
essentially
proved in \cite[Theorem~2.1]{RS}
(see also \cite[Proposition~2.1]{FI}).
Here we give a sketch of the proof 
for readers' convenience.
%by modifying the argument in \cite{FI} to $B_R$.
\begin{proof}[Proof of Proposition~\ref{Perron}]
Let us define, for any $(t,x)\in (0,T)\times B_R$ and $n\ge 1$,
$$
U^{(0)}(t,x):
={\rm e}^{t\Delta}u_0 
\ \ { \rm  and}\ \ U^{(n)}(t,x):={\rm e}^{t\Delta}u_0+
\int_0^t{\rm e}^{(t-s)\Delta}f(U^{n-1}(s)) ds.
%, \ \  n\geq 1.
$$
Thanks to the property $\bar u(t,x)\geq  {\rm e}^{t\Delta} u_0$, for $(t,x)\in (0,T)\times B_R$, and the monotonicity of $f$, by induction, we get
$$
0\leq U^{(n)}(t,x)
\leq U^{(n+1)}(t,x)\leq \bar u(t,x) 
\ {\rm for \ any }\ (t,x)\in (0,T)\times B_R \ {\rm and}\ n\geq 1.
$$
Moreover, each term  $U^{(n)}\in{C}^{1,2}((0,T)\times B_R)\cap{C}((0,T)\times \overline{B_R})$ and $U^{(n)}(t,x)=0$ for any $(t,x)\in (0,T)\times \partial B_R$. Therefore,
 $u(t,x)=\lim_{n\to \infty}U^{(n)}(t,x)$ exists and by the Monotone Convergence Theorem it  is a solution to  the integral  equation \eqref{int1}. Moreover, by the regularizing property of the heat kernel we get
  $u\in{C}^{1,2}((0,T)\times B_R)\cap{C}((0,T)\times \overline{B_R}) $. By standard computations, $u$ is also a  solution to the differential equation \eqref{diff1}.
\end{proof}

\section{Existence of a regular mild solution}

 In this Section we prove  Proposition  \ref{extcl}, that guarantees the existence of a regular mild solution to problem  \eqref{eq:1.1} with the singular solution $U$ to the associated elliptic equation as initial data.
The proof follows the same strategy  as in Theorem 2.1, 2) in \cite{IRT1}. 

\medskip
First, we  construct  a supersolution to  problem \eqref{eq:1.1}. To this end
we introduce 
a suitable auxiliary
Cauchy problem with a well--chosen polynomial nonlinearity whose
solutions, via the generalized Cole-Hopf transformation,  can be transformed to {\it supersolutions} of the Cauchy
problem \eqref{eq:1.1}. Let us first remark that
since $f'(s)F(s) \to 1^-$, as $s\to \infty$,  there exists  a positive real number $\beta$ such that 
\begin{equation}\label{beta1}
f'(s)F(s)\leq 1, \ \ \ {\rm for\ any\ } s\geq \beta.
\end{equation}
Therefore, since $(fF)'=f'F-1\le 0$,  the function $f(s)F(s)$ is decreasing for large $s$, and choosing $\beta$ large enough we also have 
\begin{equation}\label{beta2}
\beta-2f(\beta)F(\beta)\geq 0.
\end{equation}
Let us now introduce the auxiliary Cauchy problem.
In order to do so we  define
$$
w_0(x):=\left\{\begin{aligned}&\left(F(U(x))\right)^{-1/2}&&
\text{if}\ \ U(x)> \beta,\\
&\left(F(\beta)\right)^{-1/2}&&
\text{if}\ \ U(x)\leq \beta.\\
\end{aligned} \right.
$$
We remark that 
$$
w_0(x)
=\frac{2}{|x|\sqrt{B(1-2\log|x|)}}\bigl(1+o(1)\bigr)
\ \ {\rm as}\ \   x\to 0,
$$ 
by \eqref{eq:ioku1.17} in Proposition \ref{est}.

Therefore, the function  $w_0$ is not in  $L^{2}$,  but it belongs to any Lorentz space $L^{2,q}$, with $q>2$.
Now, we consider the Cauchy problem
\begin{equation}\label{eq2}
\left\{
\begin{split}
&\partial_tw-\Delta w
=\frac{w^3}{2}
&& \text{in}\ \ B_{R}, \ t>0,
\\
&w(t,x)
=F(\beta)^{-\frac 12}
&& \text{on}\ \partial B_{R}, \ t>0,\\
&w(0,x)=w_0(x).\ \ \ \ \
\end{split}
\right.
\end{equation}
By adapting to the framework of Lorentz spaces the standard
contraction mapping argument developed by Weissler~\cite{W1} and
Brezis--Cazenave~\cite{BC}, it is possible to obtain the following existence result, proved in \cite{IRT1}.

\begin{proposition}[{\cite[Proposition 6.1]{IRT1}}]\label{ex} Let $2<q\leq 5$. There exists a positive time $T=T({w_0})$ and a
	unique solution $w$ of the Cauchy problem \eqref{eq2} such that
	$w\in C([0,T];L^{2,q})$, $t^{3/10}w(t)\in C([0,T];L^{5})$ and
	$\lim_{t\to 0}t^{3/10}\|w(t)\|_{L^5}=0$.
	Moreover, $w\in C^{1,2}((0,T)\times B_R)\cap C((0,T)\times \overline{B_R})$
	and it is   a classical solution of \eqref{eq2} on $(0,T)\times
	B_{R}$,  $w(t,x)\geq (F(\beta))^{-\frac 12}$ for any $(t,x)\in (0,T)\times B_R$ and $w(t,x)\geq {\rm e}^{t\Delta}w_0(x)$, for any $(t,x)\in (0,T)\times B_R$.
\end{proposition}
Then, let us define
\[ \bar {u}(t,x)=F^{-1}(w^{-2}(t,x)) \ \ \ \ {\rm  for \ any \ }(t,x)\in [0,T)\times  B_R,
\]
where $F^{-1}$ is the inverse function of $F$ and $w$ is the
solution constructed in Proposition~\ref{ex}. 
Since $w(t,x)\geq \left(F(\beta)\right)^{-\frac 12}$, for any $(t,x)\in (0,T)\times B_R$, $w(t,x)= \left(F(\beta)\right)^{-\frac 12}$, for any $(t,x)\in (0,T)\times \partial B_R$  and  $F^{-1}(s)$ is decreasing, we have
$$
\bar{u}(t,x)=F^{-1}(w^{-2}(t,x))\geq F^{-1}(F(\beta))=\beta\ \ \ {\rm for \ any \ }(t,x)\in (0,T)\times  B_R
$$
and 
$$
\bar{u}(t,x)=F^{-1}(w^{-2}(t,x))= F^{-1}(F(\beta))=\beta\ \ \ {\rm for \ any \ }(t,x)\in (0,T)\times  \partial B_R.
$$
Moreover,  
$$
\bar{u}(0,x)=F^{-1}(w^{-2}_0(x))=\left\{\begin{aligned}&U(x)&&
\text{if}\ \ U(x)> \beta,\\
&\beta&&
\text{if}\ \ U(x)\leq \beta,\\
\end{aligned} \right.
$$
and so $\bar{u}(0,x)\geq U(x)$.
Now, by direct computations, we obtain
$$
\begin{aligned}
\partial_t\bar{u}-\Delta \bar{u}-f(\bar{u})
=4f(\bar{u})\, w^{-4}|\nabla w|^2 \left(\frac 32-f'(\bar{u})F(\bar{u})\right)\geq 0
\end{aligned}
$$
since $f'(\bar{u})F(\bar{u})\leq 1$ for  any $\bar{u}\geq \beta$.
Therefore, $\bar{u}\in{C}^{1,2}((0,T)\times B_R))\cap C((0,T)\times\overline{B_R})$ and it satisfies
\begin{equation}\label{eq:dif}
\left\{ \begin{array}{ll}
\displaystyle \partial_t\bar{u}\geq\Delta \bar{u}+f(\bar{u}) \ &\hbox{ in }\ B_R, \ t\in (0,T), \vspace{0.2cm} \\
\quad\bar{ u}(t,x) = \beta &\hbox{ on }\ \partial B_R, \ t\in (0,T),\vspace{0.2cm}
\\
\quad \bar{u}(0,x)\geq U(x)\ &\hbox{ in }\ B_R.
\end{array} \right.
\end{equation}
Namely, the
transformed function $\bar u$ is a supersolution to  the original 
problem \eqref{eq:1.1}.

\medskip
 Now, we would like to apply  Perron's monotone method (Proposition \ref{Perron}) to the supersolution $\bar u$, to construct  a mild  solution to  problem \eqref{eq:1.1}, that belongs to $L^\infty_{loc}(0,T;L^\infty)$ (and so it is classical on $(0,T)\times B_R$).  In order to apply Perron's monotone method,  we need to  prove  that the function $\bar{u}$ is also a supersolution to the integral equation associated to the differential equation. This is a consequence  of the 
 Lemma~2.3 in \cite{FI} (see also Remark 6, (1)) and of the following properties
of $\bar u$.

\begin{proposition}\label{prop2}
	The function $\bar{u}(t,x)=F^{-1}(w^{-2}(t,x))$ satisfies the following properties:
	
	\medskip
	\noindent i) $\bar{u}(t,x)\geq {\rm e}^{t\Delta}\bar{u}(0)$ for any $(t,x)\in (0,T)\times B_R$.	
	
	\medskip
	\noindent ii) $\displaystyle	\lim_{t\to 0}\Big\|\int_0^t{\rm e}^{(t-s)\Delta}f(\bar{u}(s))\ ds\Big\|_{L^{\infty}}=0.$
\end{proposition}
\par \medskip
The proof  of Proposition \ref{prop2} relies on the following result.
\begin{lemma}[{\cite[Lemma~3.1]{FI}}]\label{FI}
	Assume that there exists  some $s_1>0$
	such that
	$$
	f'(s)F(s)\leq 1 \ \ \ \ { for\ all\ } s\geq s_1.
	$$
	Then there exists a constant $C$ such that
	$
f(F^{-1}(t))
\leq Ct^{-1}
	$
	for all  $t<t_0=F(s_1)$.
\end{lemma}

\begin{proof} [Proof of Proposition \ref{prop2}.] In order to prove i) we recall that   $F(t)$ is decreasing on $(0,\infty)$ and so $\lim_{t\to 0^+}F(t)$ exists  (finite or infinite). We remark that $F^{-1}(t^{-2})$ is concave on 
$(F(\beta)^{-\frac 12},\infty)$,
because 
$$
\begin{aligned}
\frac{d^2}{dt^2}(F^{-1}(t^{-2}))&=-\frac 6{t^{4}}f(F^{-1}(t^{-2}))+\frac 4{t^{6}}f'(F^{-1}(t^{-2}))f(F^{-1}(t^{-2}))\\
&=\frac 4{t^{4}}\left (f'(F^{-1}(t^{-2}))t^{-2}-\frac 32\right)f(F^{-1}(t^{-2}))\leq 0\\
\end{aligned}
$$
for any $t\geq (F(\beta))^{-\frac 12}$, since $f'(s)F(s)-\frac 32\leq 1-\frac 32<0$, when $s\geq\beta$.
Let us define the function
$$
H(t)=
\left\{ \begin{array}{ll}
\displaystyle F^{-1}(t^{-2})  \ &\hbox{ for  }\  t>(F(\beta))^{-\frac 12}, \vspace{0.2cm} \\
\quad y(t) &\hbox{ for  }\ 0\leq t\leq (F(\beta))^{-\frac 12},\vspace{0.2cm}\\
\end{array} \right.
$$
where $y(t)= 2(F(\beta))^{\frac 32}f(\beta)[t-(F(\beta))^{-\frac 12}]+\beta$ is the tangent line to the function $F^{-1}(t^{-2})$
at $t=(F(\beta))^{-\frac 12}$.
We remark that $H(t)$ is concave  on $[0,\infty)$ and $H(0)=y(0)=\beta-2f(\beta)F(\beta)\geq 0$ thank to \eqref{beta2}.
Therefore, since  the function $H(t)$ is increasing and concave on $[0,\infty)$,
$w(t,x)\geq F(\beta)^{-\frac 12}$,  and  thanks to Lemma \ref{jensen},  we have
$$
\begin{aligned}
\bar u(t,x)&=F^{-1}(w(t,x)^{-2})=H(w(t,x))\geq H({\rm e}^{t\Delta}w_0(x))\\
&\geq   {\rm e}^{t\Delta}H(w_0(x))= {\rm e}^{t\Delta}F^{-1}(w_0^{-2}(x))={\rm e}^{t\Delta}
{\bar u(0).}
\end{aligned} 
$$
This concludes the proof of i).

Now property ii) is a consequence of Lemma \ref{FI}.
Indeed, since
$$
\bar{u}(t,x)=F^{-1}(w^{-2}(t,x)) \ \ \ \ {\rm and}\ \ \
\ w^{-2}(t,x)\leq F(\beta),
$$
by applying the previous
{Lemma} \ref{FI}
we get
\begin{equation}\label{eq:6.3c}
\begin{aligned}
\Big\|\int_0^t{\rm
	e}^{(t-s)\Delta}f(\bar u(s))ds\Big\|_{L^{\infty}}
&\leq 
C
\Big\| \int_0^t{\rm e}^{(t-s)\Delta}w^2(s)ds\Big\|_{L^{\infty}}\vspace{0.2cm}\\
&\leq
C\int_0^t\frac{1}{(t-s)^{2/5}s^{3/5}}\, ds\
\Big(\sup_{0<s<t}s^{3/10}\|w(s)\|_{L^{5}}\Big)^2\\
& \leq
\tilde{C}\Big(\sup_{0<s<t}s^{3/10}\|w(s)\|_{L^{5}}\Big)^2
\end{aligned}
\end{equation}
for some $C,\tilde C>0$.
This 
and 
the fact
$\lim_{t\to 0}\sup_{0<s<t}s^{3/10}\|w(s)\|_{L^{5}}=0$
from Proposition~\ref{ex} yields the conclusion.
\end{proof}

Thanks to Proposition \ref{prop2} and \cite[Lemma 2.3 and Remark 6]{FI}, we obtain that $\bar u$ is a supersolution to the integral equation \eqref{int0} with $u_0=U$. Therefore,
 $\bar{u}$ satisfies the hypothesis of Proposition \ref{Perron}, choosing $u_0=U$, and  we conclude that there exists  a function $u\in{ C}^{1,2}((0,T)\times B_R)\cap{ C}((0,T)\times \overline{B_R}) $
satisfying
$$
{u}(t,x)={\rm e}^{t\Delta}U+\int_{0}^t{\rm e}^{(t-s)\Delta}f({u}(s))\, ds
$$
and 
$$
\partial_tu=\Delta u+f(u)$$
for any $(t,x)\in (0,T)\times B_R$. Moreover,  $0\leq u(t,x)\leq\bar{u}(t,x)$ for any $(t,x)\in (0,T)\times B_R$, $u(t,x)=0$ for any $(t,x)\in (0,T)\times\partial B_R$.
%So, by the regularizing properties of the heat kernel, one may prove that 
Furthermore, by $u\in {C}((0,T)\times \overline{B_R})$,
$u\in C((0,T), L^p)$
for all $p\in [1,\infty)$, and hence $u\in C((0,T); Y)$.
Finally,
%thanks to property ii),
% if we define $u(0)=U$ 
we get
$$
\|u(t)-U\|_{L^p}\leq \|u(t)-{\rm e}^{t\Delta}U\|_{L^p}+\|{\rm e}^{t\Delta}U-U\|_{L^p}
$$ 
and both terms in the right hand side of the last inequality  tends to 0 as $t\to 0$ thanks to Proposition~\ref{prop2} and $L^\infty(B_R)\subset L^p(B_R)$. This concludes the proof of Proposition \ref{extcl}.

\section{Non-uniqueness of mild solutions}

In this Section  we prove 
{Theorem~\ref{nonun}}.
Let $u_0=U$.
Let $T>0$ and $u(t)$
be 
the existence time and regular mild solution 
which is constructed in Proposition~\ref{extcl}, respectively.  
We prove that $U$ is a mild solution with $u_0=U$.
Since $U\in Y$ and it does not depend on $t>0$, 
we have $U\in C([0,T];Y)$.
It remains to prove that $U$ satisfies the integral equation~\eqref{eq:6.1}. 

In what follows we use the following notation: 
for a function $K=K(t,x,y):(0,\infty)\times B_R\times B_R\to \mathbb{R}$ we write
\[
\nabla_y K:=\left(\frac{\partial K}{\partial y_1},\frac{\partial K}{\partial y_2}\right),
\ \ 
\Delta_y K:=\frac{\partial^2K}{\partial y_1^2} + \frac{\partial^2 K}{\partial y_2^2},
\ \ 
\Delta_x K:=\frac{\partial^2K}{\partial x_1^2} + \frac{\partial^2 K}{\partial x_2^2}
\]
and
\[
K\in L^{\infty}_y\ \ \text{if $K(t,x,\cdot)\in L^{\infty}(B_R)$
for any $t>0$ and $x\in B_R$},
\]
where $x=(x_1,x_2), y=(y_1,y_2)$.

Fix arbitrary $t>0$. 
It holds 
for any $0<s<t$
that
\begin{equation}
\label{eq:6.1a}
\frac{d}{ds}\left(\textrm{e}^{(t-s)\Delta}U\right)
=
-\Delta \Bigl(\textrm{e}^{(t-s)\Delta}U\Bigr)
=
-\Delta_x \int_{B_R}G(t-s,x,y)U(y)dy,
\end{equation}
where $G:(0,\infty)\times B_R\times B_R \to \mathbb{R}$ is the Dirichlet heat kernel in $B_R$.
Recall that $G$ is a positive $C^{\infty}$ function, $G(t,x,y)=G(t,y,x)$, $G(t,x,\cdot)$ and $\Delta_x G(t,x,\cdot)$ are bounded in $B_R$ for any fixed $x\in B_R,\ t>0$ 
(see \cite[{(16.16) and Theorem~16.3 in p.413}]{LSU}).
Thus, we have
\begin{equation}\label{eq:6.2}
\begin{aligned}
-\Delta_x \int_{B_R}G(t-s,x,y)U(y)dy
&
=
-\int_{B_R}\Delta_x G(t-s,x,y)U(y)dy
\\
&
=
-\int_{B_R}\Delta_y G(t-s,x,y)U(y)dy.
\end{aligned}
\end{equation}
Fix $\delta\in (0,R)$. 
Since $U\in C^2(B_R\setminus B_{\delta})$
satisfies $-\Delta U=f(U)$ in $B_R\setminus B_{\delta}$, 
and $U=0$
on $\partial B_R$, $\left.G(t,x,y)\right|_{y\in \partial B_R}=0$,
we have, by integration by parts, that
\begin{equation}
\label{eq:6.3}
\begin{aligned}
-\int_{B_R\setminus B_{\delta}}\Delta_y G(t-s,x,y)U(y)dy
=
&
\int_{B_R\setminus B_{\delta}}G(t-s,x,y)f\bigl(U(y)\bigr)dy
\\
&
-
\int_{\partial B_{\delta}}\Bigl(\nabla_y G\cdot \nu\Bigr) U(y)dy
+
\int_{\partial B_{\delta}}\Bigl(\nabla U\cdot \nu\Bigr) G(t,x,y)dy,
\end{aligned}
\end{equation}
where $\nu$ is the outer normal vector on $\partial B_{\delta}$.
Applying Proposition~\ref{est} and $G,\nabla_y G\in L^{\infty}_y$, we have
\begin{equation}
\label{eq:6.4}
\begin{aligned}
&
\left|\int_{\partial B_{\delta}}\Bigl(\nabla_y G\cdot \nu\Bigr) U(y)dy\right|
\le
C \delta (-2\log \delta)^{1-\frac{1}{B}+\sigma} \to 0\ \ \text{as}\ \delta \to 0,
\\
&
\left|\int_{\partial B_{\delta}}\Bigl(\nabla U\cdot \nu\Bigr) G(t,x,y)dy\right|
\le
\frac{C}{(-2\log \delta)^{\frac{1}{B}-\sigma}} \to 0\ \ \text{as}\ \delta \to 0.
\end{aligned}
\end{equation}
On the other hand, it follows from $G,\Delta_yG\in L^{\infty}_y$ and $U,f(U)\in L^1(B_R)$ by Remark~\ref{est}
that
\begin{equation}
\label{eq:6.5}
\begin{aligned}
&
\int_{B_R\setminus B_{\delta}}\Delta_y G(t-s,x,y)U(y)dy
\to
\int_{B_R}\Delta_y G(t-s,x,y)U(y)dy\ \ \text{as}\ \delta \to 0,
\\
&
\int_{B_R\setminus B_{\delta}}G(t-s,x,y)f\bigl(U(y)\bigr)dy
\to
\int_{B_R}G(t-s,x,y)f\bigl(U(y)\bigr)dy\ \ \text{as}\ \delta \to 0.
\end{aligned}
\end{equation}
Combining \eqref{eq:6.1a}--\eqref{eq:6.5}, we obtain that
\begin{equation}
\label{eq:6.6}
\begin{aligned}
\frac{d}{ds}\left(\textrm{e}^{(t-s)\Delta}U\right)
&
=
-\Delta_x \int_{B_R}G(t-s,x,y)U(y)dy
\\
&
=
\int_{B_R}G(t-s,x,y)f\Bigl(U(y)\Bigr)dy
\\
&
=
\textrm{e}^{(t-s)\Delta}f(U)
\end{aligned}
\end{equation}
for all $s\in (0,t)$.

Let $\epsilon\in (0,t)$. Integrating the both side of \eqref{eq:6.6} on $(\epsilon,t-\epsilon)$,
we have
\[
\textrm{e}^{\epsilon \Delta}U
=
\textrm{e}^{(t-\epsilon) \Delta}U
+
\int_{\epsilon}^{t-\epsilon}
\textrm{e}^{(t-s)\Delta}f(U)
ds.
\]
It follows from $U\in Y$ 
that 
\[
\|\textrm{e}^{\epsilon\Delta}U-U\|_{L^p}\to 0,\quad
\|\textrm{e}^{(t-\epsilon)\Delta}U-\textrm{e}^{t\Delta}U\|_{L^p}\to 0
\]
as $\epsilon\to 0$ for any $p\in [1,\infty)$. 
Furthermore, it holds that
\[
\begin{aligned}
&
\left\|
\int_{\epsilon}^{t-\epsilon}
\textrm{e}^{(t-s)\Delta}f(U)
ds
-
\int_{0}^{t}
\textrm{e}^{(t-s)\Delta}f(U)
ds
\right\|_{L^p}
\\
&
\le
\left(
\int_{0}^{\epsilon}
(t-s)^{-1+\frac{1}{p}}
ds
+
\int_{t-\epsilon}^t
(t-s)^{-1+\frac{1}{p}}
ds
\right)
\left\|
f(U)
\right\|_{L^{1}}
\to 0\ \ \text{as}\ \epsilon \to 0.
\end{aligned}
\]
Hence, $U$ satisfies the integral equation~\eqref{eq:6.1}.
This completes the proof of Theorem~\ref{nonun}.

\bigskip
\bigskip

\noindent
{\bf Acknowledgments.}

This work was partially funded by JSPS KAKENHI (grant number 21K18582, 23K03179)  and it has been written within the activities of GNAMPA group (Gruppo Nazionale per l'Analisi Matematica, la Probabilit\`a e le loro Applicazioni) of INdAM (Istituto Nazionale di Alta Matematica), Italy.

\end{document}